 \documentclass[12pt]{amsart}

 \usepackage{pifont,geometry}
 \usepackage{amsmath,amssymb,amsthm}
 \usepackage[latin1]{inputenc}
 \usepackage[english]{babel}
 
\usepackage{algorithm,algpseudocode}
 \usepackage{graphicx}
 \usepackage{enumitem}
 
 \usepackage[x11names,rgb,table]{xcolor}
 \usepackage{tikz}
 \usetikzlibrary{arrows,shapes}

 \usepackage[pdftex,bookmarks,colorlinks]{hyperref}

 \newtheorem{theorem}{Theorem}[section]
 \newtheorem{corollary}[theorem]{Corollary}
 \newtheorem{proposition}[theorem]{Proposition}
 \newtheorem{lemma}[theorem]{Lemma}
 
 \theoremstyle{definition}
 \newtheorem{definition}[theorem]{Definition}
 \newtheorem{example}[theorem]{Example}
 \newtheorem{remark}[theorem]{Remark}
\newtheorem{convention}[theorem]{Convention}

\numberwithin{equation}{section}

\theoremstyle{theorem}
\newtheorem*{teo}{Theorem}

 \newcommand{\FQ}{\mathbb{F}_{q^2}}
 \newcommand{\Fq}{\mathbb{F}_q}
 \newcommand{\NN}{\mathbb{N}}
 
 \newcommand{\bc}{\mathbf{c}}
 
 \newcommand{\Tra}{\mathrm{Tr}}
 
 \newcommand{\No}{\mathrm{N}}
 \newcommand{\rk}{\mathrm{rk}}
 
 \newcommand{\In}{\mathrm{In}}
 \newcommand{\LM}{\mathrm{LM}}
 \newcommand{\supp}{\mathrm{Supp}}
 
 \newcommand{\He}{\mathcal{H}}
 \newcommand{\xx}{\mathcal{X}}
 \newcommand{\y}{\mathcal{Y}}

 \newcommand{\cN}{\mathcal{N}}
 \newcommand{\B}{\mathcal{B}}
 
 \def\ecr{\color{black}}

\begin{document} 	
 	\title{Hermitian codes and complete intersections}
 	
 	\author{Chiara Marcolla}
 	\address{\textnormal{Chiara Marcolla}. Dipartimento di Matematica dell'Universit\`{a} di Torino\\ 
 		Via Carlo Alberto 10, 
 		10123 Torino, Italy}
 	\email{{chiara.marcolla@gmail.com}}

 	\author{Margherita Roggero}
 	\address{\textnormal{Margherira Roggero}. Dipartimento di Matematica dell'Universit\`{a} di Torino\\ 
 		Via Carlo Alberto 10, 
 		10123 Torino, Italy}
 	\email{{margherita.roggero@unito.it}}

 	\keywords{Hermitian code, minimum-weight codeword, complete intersection}
 	\subjclass[2010]{11G20,11T71} 
 	
 	
 	\begin{abstract}

In this paper we consider the Hermitian codes defined as the dual codes of one-point evaluation codes on the
Hermitian curve $\mathcal H$ over the finite field $\FQ$. We focus on those with  distance $d \geq q^2-q$ and  give a geometric description of the support of  their  minimum-weight codewords. \\
We consider the unique writing $ \mu q + \lambda (q+1)$ of the distance $d$ with  $\mu, \lambda$ non negative integers, and $\mu \leq q$, and  consider all the  curves  $\mathcal X$  of the affine plane $\mathbb A^2_{\FQ}$  of degree $\mu+\lambda$  defined by  polynomials  with $x^\mu y^\lambda$ as  leading  monomial   w.r.t. the $\texttt{DegRevLex}$ term ordering (with $y>x$).   We  prove that  a zero-dimensional subscheme $Z$ of $\mathbb A^2_{\FQ}$ is the support of  a  minimum-weight codeword  of the Hermitian code with distance $d$ if and only if it  is made of $d$   simple $\FQ$-points and   there is a curve $\mathcal X$ such that   $Z$ coincides with the scheme theoretic intersection $\mathcal H \cap  \mathcal X$ (namely,  as  a cycle,   $Z=\mathcal H \cdot \mathcal X$).
Finally, exploiting  this geometric characterization, we propose an algorithm to compute the number of minimum weight codewords and  we present comparison tables between our algorithm and  MAGMA command \texttt{MinimumWords}.
\end{abstract}
\maketitle


\section{Introduction}

Let $q$ be a power of a prime. The \textit{Hermitian curve} $\He$ is the affine, plane curve defined  by the polynomial $x^{q+1}=y^q+y$. It is a smooth curve of genus $g=\frac{q^2-q}{2}$ with only one point at infinity.
The curve $\He$ is one of the best known example of  \emph{maximal curve}, that is with  the  maximum number of $\FQ$-points allowed by the Hasse-Weil  bound \cite{CGC-alg-art-rucsti94}. 

Starting from the Hermitian curve and any positive integer $m$, it is possible to construct the geometric one-point Goppa code $C_m$ on $\He$, that is called \textit{Hermitian code}. This  is by far the most studied among geometric Goppa codes, due to  the  good properties of the Hermitian curve. 

 For a thorough exposition of the main features of Hermitian curves and codes we refer to  \cite{CGC-cod-book-hirschfeld2008algebraic} and to  Section 8.3 of \cite{CGC-cd-book-stich}.

\medskip

 In 1988, Stichtenoth \cite{stichtenoth1988note}  describe generator and parity-check matrices of  Hermitian codes, introduced one year before by van Lint, Springer \cite{van1987generalized} and Tiersma \cite{tiersma1987remarks}.  Stichtenoth, for any $m > q^2-q-2$, finds a formula for the distance $d$ of $C_m$. A few years later,  Yang and Kumar \cite{yang1992true}   bring to completion Stichtenoth work finding  the distance of  the remaining codes  $C_m$.  Finally, in 1995
 Kirfel and Pellikaan \cite{kirfel1995minimum} presented a different and much shorter proof of these results using linear algebra and the theory
 of semigroups. Since then, the results were extended
 to two-point Hermitian codes by Homma and Kim  \cite{homma2005toward,homma2006complete,homma2006two,homma2006twob} exploiting a similar proof to \cite{yang1992true}.
  After that, in 2010 Park \cite{park2010minimum}, using a method based on \cite{kirfel1995minimum},  gives a short and easy proof of the same results of Homma and Kim and obtains a geometrical characterization of minimum weight codewords as multiplications of conics and lines. 
  The description of minimum-weight codewords of some  two-points Hermitian codes have been found three years later by Ballico and Ravagnani \cite{ballico2013dual} using different techniques. 
Moreover,  Yang \cite{yang1994weight} and Munuera \cite{munuera1999second} obtain   the values of many \textit{generalized Hamming weights} also called \textit{weight hierarchies}. Finally, Barbero and Munera  \cite{barbero2000weight}   find the complete sequence of weight hierarchies of Hermitian codes by an exhaustive computation of the bounds given by Heijnen and Pellikaan \cite{heijnen1998generalized}.
Other results concerning generalized Hamming weights of codes on Hermitian curve and the \textit{relative} generalized Hamming weights  can be found i.e. in \cite{ballico2016higher,geil2014relative,homma2009second,lee2015bounds}.\\

In \cite{CGC-cd-book-AG_HB} the Hermitian codes are seen as a  sub-family of evaluation codes. Using this  different approach, the authors   divide the  codes $C_m$ in four \textit{phases} with respect to the integer $m$, and  for each of them give 
explicit formulas linking  dimension and  distance. In this paper we adopt  their classification of Hermitian codes in four phases (with minor changes,  as summarized in Table~\ref{Tab}) and,  as in \cite{CGC-cd-book-AG_HB}, we consider the Hermitian code as the dual of one-point evaluation code 
(Definition \ref{def:codice}).\\ 

Later on, the  research about Hermitian codes branches out in several different lines. Some  papers, as for instance \cite{lee2009list,lee2010algebraic,CGC-cd-prep-manumaxchiara12,o2000decoding}, deal with the problem of finding efficient algorithms for the decoding of the Hermitian codes. 

An hard problem is that of determining  the weight distribution, in particular the small-weight distribution. So far, few partial results are known, and the first of them   appears only in 2011 (\cite{CGC-cod-art-marcpell2011weights}).\\
The geometric characterization of the small-weight  codewords of the Hermitian code $C_m$ for  a few  cases of $m$ (mainly in the first and second phase)  can be found in \cite{CGC-cod-art-ballico2012geometry,CGC-cod-art-couvreur2012dual, CGC-cod-art-fontanari2011geometry,CGC-cd-art-marcolla2015small, CGC-cod-art-marcpell2011weights}. 
In particular in \cite{CGC-cod-art-marcpell2011weights}    and   \cite{CGC-cd-art-marcolla2015small}, the first author of this paper and her co-authors  study Hermitian codes $C_m$ with distance $d\leq q$, that is with  $m\leq q^2-2$ (first phase). They prove that  the points of the support of  any minimum-weight codeword of $C_m$ lie in the intersection between a line and $\He$;  
on the other hand, any  set of $d$ points   in such a \emph{complete  intersection}   corresponds to   minimum-weight codewords.
This characterization allows the authors to compute  the number of minimum-weight  codewords.

In 2012,  Couvreur   \cite{CGC-cod-art-couvreur2012dual}  investigate the minimum-weight codewords problem for codes over an affine-variety $\xx $  by a new method.  Quoting Couvreur paper \emph{the approach is based on problems \'a la Cayley-Bacharach and consists in describing the minimal configurations of points on $\xx $ which fail to impose independent conditions on forms of some degree.}

As an application of this approach,  in \cite{CGC-cod-art-fontanari2011geometry} the authors find  a geometric characterization of small-weight codewords of $C_m$ for some $m$ and $d\leq 3q-6$ (first and second phase). In particular they
prove that the set of  points that are the support of a minimum-weight codeword (or a subset of them) is a cut on  $\He$ by  either a line, or a conic, or the complete intersection of two curves of degrees $q-2$ and $3$. A similar result is found for all codewords of the first phase having weight $v\leq 2d-3$.

A new proof of the above results and some new information about the small-weight codewords of Hermitian codes $C_m$ with $m \leq q^2+q$ and $d=2q+2,2q+1,2q$ or $d\leq q$ are presented in  \cite{CGC-cod-art-ballico2012geometry}.

So, up to now, the geometric descriptions  of minimum weight codewords that we  can find in literature  concern only  a  few Hermitian codes, mainly those in the first and second phases  (namely, with $m\leq 2q^2-2q-3$): note that, for each value of $q$,  these two phases cover only  a small fraction  (about  $\frac2q$) of all the Hermitian codes.

In this paper we provide a geometric characterization for minimum-weight codewords of  all the  Hermitian codes $C_m$   in the remaining two  phases. Our main result is the following theorem.

\begin{teo}[Theorem \ref{curvagenerica}] 
Let $C_m$ be an Hermitian code with $m\geq 2q^2-2q-2$ and distance $d=m-q^2+q+2$. Let $\lambda, \mu$ be the unique pair of non-negative integers such that $d=\mu q +\lambda (q+1)$ and $\mu \leq q$ and let   $D$ be a divisor on the Hermitian curve $\He$  made of simple points with coordinates in $\FQ$.\\
Then $D$ is the support of  a minimum-weight codeword of $C_m$  if and only if   it is  the complete intersection of $\He$  and a curve $\xx$    defined by a polynomial  $F$ 
of the following type:  
\begin{equation} \label{desF} F= x^\mu y^{\lambda}+ \Sigma  a_i x^{r_i} y^{s_i} \hbox{ where  either  }  r_i+s_i< \mu + \lambda \hbox{  or }   r_i+s_i= \mu + \lambda  \hbox{ and }r_i <\mu .\end{equation}
\end{teo}

 Besides the theoretical interest of this geometric description, we note that all  the properties that are  involved can be easily checked through standard tools of computational algebra,  as for instance those concerning the  ideal membership. In fact we can reformulate the above Theorem as follows (Corollary \ref{corollfin})

\medskip

{\it A zero-dimensional subscheme  $D$  of $\He$ is the support of  a minimum-weight codeword of $C_m$ if and only if the ideal $I_D$ of $D$  is generated by $H$ and a polynomial   $F$   as in \eqref{desF} and it contains  $x^{q^2}-x$ and $  y^{q^2}-y$.}
\medskip

 Exploiting this explicit  characterization,  we propose  Algorithm~\ref{Alg} to compute the number of minimum weight codewords. We implemented this algorithm with MAGMA software and we compare its performance with those of the command \texttt{MinimumWords} already present in MAGMA  for some codes of the third and fourth phase (i.e. Hermitian codes with $m\geq 2q^2-2q-2$) with $q=3$ and $q=4$.
Depending on the code the two algorithm have different performances, in general our algorithm is more convenient for lower $m$ (see Table \ref{Tab.q3.3fase}). A striking case is that with  $q=4$ and $m=22$: our algorithm  computes the number of minimum-weight codewords  in $85.87$ seconds  (they are  $150000$), while \texttt{MinimumWords}  declares that termination requires $10^8$ years.\\

To achieve these results we use a different term ordering with respect to what can be usually found in literature, that is the $\mathtt{DegRevLex}$ with $y>x$. 
We exploit tools such as the Hilbert polynomial and the so called sous-\'escalier. This last tool, together with Gr\"obner basis, has already been exploited in a slightly different way in \cite{geil2008evaluation,geil2000footprints,geil2017bounding} to establish or to determine some affine-variety code parameters.

We also point out that analogous results for codes of the other phases, namely  codes  $C_m$  with  $m<2q^2-2q-2$,  has been obtained by the authors in \cite{marcollaroggero2016minimum}. 
Finally, a generalization to codewords of small weight is in progress and we are confident that, from this strong geometric characterization, also  the explicit computation of the weight distribution will follow.\\ \ecr

The paper is organized as follows:
\begin{itemize}
%
%
\item In Section \ref{Sec.pre} we recall some basic definitions  about Hermitian curves and codes and we discuss our non-standard choice of  the term ordering. 

\item In Section \ref{Sec.prere} we focus on the  Hermitian codes $C_m$ with $ m\geq q^2-1$ and prove some preliminary results  that we will exploit in the following sections. 

\item In Section \ref{Sec.dist} we study the Hermitian codes $C_m$ with $ m\geq 2q^2-2q-2$.  The main result of this section is Theorem~\ref{diseguaglianza} which will be key tools in the final section. In some sense it generalize the classical results by Stichtenoth \cite{stichtenoth1988note} about the distance formula. Indeed, in Corollary \ref{cor.dist} we recover this same formula  as a special case of what is proved in Theorem~\ref{diseguaglianza}.

\item In Section \ref{Sec.MinWord} we state and prove Theorem \ref{curvagenerica}, which gives a geometric characterization for any minimum-weight codewords of the Hermitian codes $C_m$ with $ m\geq 2q^2-2q-2$.
\item In Section \ref{Sec.alg} we propose Algorithm \ref{Alg} to compute the number of minimum-weight codewords and we present comparison tables (Tables \ref{Tab.q3.3fase},\ref{Tab.q4.3fase}) with MAGMA command \texttt{MinimumWords}.
\item At the end  we draw the conclusions.
\end{itemize}


\section{Generalities and introduction to codes} \label{Sec.pre}

\subsection{The geometric setting}

Let $\FQ$ be the finite field with $q^2$ elements, where $q$ is a power of a prime, and let $K$ be its algebraic closure.
For any ideal $I$ in the polynomial ring $\FQ[x,y]$ 
we denote by $\mathcal{V}(I)$  the corresponding subscheme of $\mathbb A_{\FQ}^2$;   if $g_1,\ldots,g_s\in\FQ[x,y]$, we denote by $\langle g_1,\ldots,g_s\rangle$ the ideal they generate.  If $Y$ is a subscheme of  $\mathbb A_{\FQ}^2$, we will denote by $I_Y$ the ideal such that $Y=\mathcal V(I_Y)$ and by   $A_{Y}$ the coordinate ring  of $Y$, namely the quotient   $\FQ[x,y]/I_Y$. \\ 
 Though we are mainly interested in the closed points with coordinates in $\FQ$ ($\FQ$-points for short), we will  not identify a subscheme of $\mathbb A_{\FQ}^2$  with the set of its  $\FQ$-points, or, more generally,   a scheme with its support.   For the sake of simplicity, we sometimes  consider the  reduced schemes  as the set of their irreducible components; for instance a  zero-dimensional and reduced scheme   $Z$, can be considered as   the set of its (closed) points  and denoted $Z=\{Q_1, \dots, Q_s\}$.  

If $Z$ is a zero-dimensional subscheme of  $\mathbb A_{\FQ}^2$,  we will say that it is a {\it divisor}  on a curve $\mathcal X$ if $I_{\mathcal X} \subset I_Z$. Moreover, we will say that $Z$  it is a {\it complete intersection}  of $\mathcal X$ with another curve $\mathcal Y$ if $I_Z=I_\mathcal X+I_\mathcal Y$; in this case  we can also write $Z$ as $\sum m_i Q_i$ where  $m_i$ is  the intersection multiplicity of $\mathcal X$ and $\mathcal Y$ at $Q_i$.

\subsection{The Hermitian curve}

The \textit{Hermitian curve} $\mathcal H$ is the curve in the  affine plane $\mathbb{A}_{\FQ}$ defined by the polynomial $H:=x^{q+1}-y^q-y$. We will denote by  $I_\He$   the ideal in $ \FQ[x,y]$  generated by $H$ and by   $A_\He$ the coordinate ring $\FQ[x,y]/I_\He$ of $\He$.   

The curve $\He$ has genus $g=\frac{q(q-1)}{2}$ and  contains exactly $n:=q^3$  $\FQ$-points that we will denote by $P_1, \dots, P_n$  \cite{CGC-alg-art-rucsti94}.   We will always denote   by   $E$  the reduced zero-dimensional scheme  $\{P_1, \dots, P_n\}$, or, equivalently, the divisor  $\Sigma P_i$ on $\He$.

The projective closure $\overline \He$ of $\He$ in $\mathbb P^2_K$ contains   only one more point   $P_{\infty}=[0:0:1]$, which is simple and  $\FQ$-rational, so that $\overline \He$  has $q^3+1$  $\FQ$-points.

We recall that the \textit{norm} $\mathrm{N}$ and the \textit{trace}
	$\mathrm{Tr}$ are two functions from $\mathbb{F}_{q^2}$ to
	$\Fq$ such that
	$\mathrm{N}(x)=x^{1+q} \mbox{ and }
	\mathrm{Tr}(x)=x+x^{q}.$
Observe that  $\He=\{\No(x)=\Tra(y)\mid x,y\in\FQ\}$.\\

\begin{remark}\label{griglia}
In the proof of the key results of this paper we will exploit  the beautiful arrangement of the set of  $\FQ$-points of the Hermitian curve, due to its {\it norm-trace} equation. There are $q^2$ vertical lines, each containing $q$ points of $E$, and there are $q^2-q+1$ horizontal lines  each containing $q+1$ points of $E$. 
See \cite{CGC-alg-book-hirschfeld1998projective} or Section~2.3 of \cite{CGC-cd-art-marcolla2015small}.
\end{remark} 
\begin{definition}  \label{defdivisori}  An  \textbf{$\FQ$-divisors on  the Hermitian curve} is a divisor $D=\sum_{i=1}^\delta Q_{i}$ where the $Q_{i}$'s are pairwise   distinct $\FQ$-points of $\He$. We will denote by   $\vert D\vert$   the degree $\delta $ of $D$. We can also write $D=\{Q_1, \dots, Q_\delta\}$. 
\end{definition}

 Important examples of $\FQ$-divisor over the Hermitian curve  are $E=\{P_1, \dots, P_n\}$, which is  the largest  one  (so that  $D$ is  a  $\FQ$-divisor on $\He$  if and only if   $D\subseteq E$) and those cut on $\He$ by any line $\mathcal L_i$, $\mathcal L_j'$ or Remark \ref{griglia}. 

\begin{remark}\label{rm:eqcampo}  We recall that, in the above notations, the ideal  $I_E$ is generated by $  H,x^{q^2}-x, y^{q^2}-y $.  
Moreover,  a zero-dimensional subscheme $D$    of $\mathbb A^2_{\FQ}$ is a  $\FQ$-divisor on $\He$ if and only if  $I_D$ contains $I_E$  or, equivalently,  if and only if  $A_{D}$ is a quotient of  $A_{E}=\FQ[x,y]/I_E$. 
Therefore,    $D$  is a  $\FQ$-divisor on $\He$ if and only if  $I_D$ contains the polynomials  $  H,x^{q^2}-x, y^{q^2}-y $  (see for instance  \cite{CGC-alg-art-seidenberg1}). 
%
We will exploit this characterization of $\FQ$-divisors   in the construction of Algorithm~\ref{Alg}.

\end{remark}

\subsection{A quick sketch on the affine--variety codes}

Let $C$ be a linear code over $\FQ$ with \textit{generator matrix} $G$.\\
We recall that the \textit{dual code} $C^{\perp}$ of $C$ is formed by all vectors $\mathbf v$ such that $G\mathbf v^T=0$ and a generator matrix of $C^\perp$ is called a \textit{parity-check matrix} of the code $C$. Moreover if $\mathbf c=(c_1, \dots,c_n)\in C$ is a codeword, then the \textit{weight} of $\mathbf c$ is the number of $c_i$ that are different from $0$, whereas, its {\it support}   $\supp(\mathbf c)$  is the set of indices  corresponding to the non-zero entries. In our case, the entries of a codeword $\mathbf c$ are labeled by the $\FQ$-points of the Hermitian curve and we will identify its support   with   the  $\FQ$-divisor $D$ which  is the  set of points that correspond to the non-zero  entries:  we will say for short that 
\textit{$D$ is the divisor on $\He$ corresponding to the codeword $\mathbf c$}.
 Finally, the {\it distance} of a code is the minimum weight of its non-zero codewords.\\

We now briefly recall the definition of affine--variety codes of which the Hermitian codes are  special cases. For more results on affine-variety codes see \cite{geil2008evaluation}.\\ 

Let us consider an ideal $I\subset \FQ[x,y]$ such that $\{x^{q^2}-x,y^{q^2}-y\}\subset I$. Then $I$ is zero-dimensional and radical (Remark \ref{eqcampo}). If  $\mathcal{V}(I) = \{Q_1, \dots, Q_r\}$,  
 the \textit{evaluation map} $\phi_I$ is defined in the following way:
\begin{equation}\label{eval}
  \begin{array}{rccl}
    \phi_I :R= & \FQ[x,y]/I & \longrightarrow & (\FQ)^r\\
    & f & \longmapsto & (f(Q_1),\dots,f(Q_r)).
  \end{array}
\end{equation}

\begin{definition}
Let $L\subseteq R$ be an $\FQ$-vector subspace of $R$.
 The {\bf affine--variety code} $C(I,L)$ is the image $\phi_I(L)$ and the
affine--variety code $C(I,L)^{\perp}$ is its dual code.
\end{definition}

We fix an order for the points $Q_1, \dots, Q_r$ and an ordered  basis 
$f_1,\dots,  f_s$  of $L$. Then $C(I,L)$ is given by the matrix $G$, called \textit{generator matrix} of the code, having in the $(i,j)$ -position the evaluation of $f_i$ in the point $Q_j$.\\

 In this paper, the Hermitian codes are seen as dual codes of special subspaces.  We associate to the Hermitian curve 
 $\He$ defined by $H=x^{q+1}-y^q-y$ the weight vector $w:=[w(x)=q, w(y)=q+1]$, so any monomial $x^ry^s$ has weight-degree $w(x^ry^s)=rq+s(q+1)$. 
 For any fixed  positive integer $m$,  we denote by $V_m$ the subspace of $A_E$ generated by  (the classes of)  the monomials of weight degree less or equal than $m$. We denote by $C_m$ the corresponding  Hermitian code (that is a dual code).
Note that,  usually, the set of monomials having weight degree $\leq m$ are not linearly independent. For this reason, we  select a  suitable set  of monomials $\mathcal B$ (a basis for $A_E$) such that, for every $m$,   $\mathcal B \cap V_m$ is a basis for $V_m$. In this way the parity-check matrix of any code $C_m$ has exactly $n-k$ rows where $k$ is the dimension of the code. To choose this basis it is convenient to use a term ordering.%

\subsection{Term ordering and  Hermitian codes}

We recall that for any given term ordering $\prec$ in a polynomial ring $S$  every polynomial
$F\in S$ has a unique \textit{leading monomial} $\LM_\prec(F)$, that is the $\prec$-largest monomial which occurs with nonzero coefficient in the expansion of $F$. If $I$ is an ideal in $S$,  the \textit{initial ideal} $\In_\prec(I)$ is the ideal generated by the leading monomials of all the polynomials in $I$, that is 
$$\In_\prec(I) =\langle LM_\prec(F) \mid F\in I \rangle.$$

The two ideals $I$ and  $\In_\prec(I)$ share a same monomial basis of the quotients given by the \textit{sous-\'escalier} $\cN(I)$ of $\In_\prec(I)$, which is the set all monomials that do not belong to $\In_\prec(I)$. \\

In this paper  we  use the \textit{graded reverse lexicographic ordering} $\mathtt{DegRevLex}$.  We recall that in the  polynomial ring $k[z_1, \dots, z_n]$,  with variables ordered as   $z_1< \dots <z_n$,  
we have  $z_1^{r_1}\cdots z_n^{r_n} <_{\mathtt{DegRevLex}} z_1^{s_1}\dots z_n^{s_n}$, if ether $\Sigma r_i<\Sigma s_i$ or $\Sigma r_i=\Sigma s_i$ and  $r_i > s_i$ with $i$ minimum index   such that $r_i\neq s_i$.

\begin{remark}  As well known,   we can read many important  geometric features of the   scheme (either affine or projective) defined by an ideal $I$  through those of the scheme defined by  the initial ideal  $\In_\mathtt{DegRevLex}(I)$, while the use of  other term orderings can  involve a significant lost of geometric information.   In fact, the ideals $I$ and $\In_\mathtt{DegRevLex}(I)$ share the same   Hilbert polynomial, so that they also share all the geometric invariants encoded in this   polynomial:  dimension,  degree, algebraic genus in the case of curves, and so on. 

  The term ordering we usually find in  literature about Hermitian codes is the weighed term ordering  $<_w$ associated to the weight vector  $w=[q ,q+1]$ (and $\mathtt{Lex}$ with $x<y$ as a \lq\lq tie-breaker\rq\rq).  More precisely:
$$
x^ry^s <_w x^{r'}y^{s'} \iff \left\{\begin{array}{l}
   rq+s(q+1)< r'q+s'(q+1) \hbox{  \ \ or }\\
  rq+s(q+1) = r'q+s'(q+1) \mbox{ and }  s<s'
\end{array}\right.$$ 

Our choice of a  different term ordering, and in particular the choice of $\mathtt{DegRevLex}$ is due to the following two reasons.

First,  the leading monomial  w.r.t. $\mathtt{DegRevLex}$  of the polynomial $H$ of  the Hermitian curve $\He$  is the  monomial $x^{q+1}$ that defines a (non-reduce) curve of the same degree as $\He$ (while  $LM_{<_w}(H)=y^q$ has a lower degree).   Second, in this paper we deal with  curves in the affine plane and also with their    projective closure   in the projective plane; the term ordering  $\mathtt{DegRevLex}$ does not substantially modify the leading monomial of any polynomial when we homogenize it, hence it is the more convenient  when  we  compute  monomial bases for the coordinate ring of  an affine  scheme and that of its projective closure  through  initial ideals and   sous-escaliers.

\end{remark}

We now compute  the  monomial   bases  of $A_E$  that can be obtained  using every  term ordering,  and in particular those  quoted above.

\begin{lemma}\label{eqcampo}      
If  $\prec$ is  any term ordering in $\FQ[x,y]$,  the  initial ideal $\In_\prec(I_E)$   is either $J_1=\langle y^{q}, x^{q^2}\rangle$ or  $J_2=\langle x^{q+1}, xy^{q^2-q},y^{q^2} \rangle$.\\
 Then, a monomial basis for $A_E$ is ether one of the following two:
\begin{equation}\label{baseB1} 
 \B_{J_1}=\{ x^ry^s \mid r\leq q^2-1,  s\leq q-1\} 
\end{equation}
\begin{equation}\label{baseB}
  \B_{J_2}=\{x^ry^s \  \mid \ 0\leq r\leq q , \   s\leq q^2-q-1\} \cup \{y^s \  \mid   \    q^2-q\leq s\leq q^2-1\}.
\end{equation}
In particular, we obtain \eqref{baseB1} with $<_w$ and \eqref{baseB} with $\mathtt{DegRevLex}$.

\end{lemma}
\begin{proof}
Depending on the term ordering, the leading monomial of $H$ is ether $y^q$ or $x^{q+1}$; the leading monomials of the field equations are in every case $x^{q^2}$ and $y^{q^2}$.

If the leading monomial of $H$ is $y^q$ then $\In_\prec(I_E)$ contains $J_1$ and the set of monomials that do not belong to $\B_{J_1}$ generate $A_E$. Therefore $\In_\prec(I_E)=J_1$  since $I_E$ and $J_1$ define zero-dimensional schemes in $\mathbb{A}^2_{\Fq}$ of the same length $q^3$.
This is the case that happens when we use $<_w$.\\
Now we assume that the leading monomial of $H$ is $x^{q+1}$. 
We observe that if $a,b$ are elements of a ring, then $a^{q-1}-b^{q-1}$ is divisible by $a-b$. 

Then $K=(x^{q+1})^{q-1} - (y^q+y)^{q-1}$ belongs to $I_E$, since it is a multiple of $H$.  Hence, $xy^{q^2-q}=\LM_\prec((x^{q^2}-x)-xK)=\LM_\prec(x(y^{q}+y)^{q-1}-x)\in \LM_\prec(I_E)$.  Therefore $\In_\prec(I_E)$ contains $J_2$ and the same argument as above  show that they coincide. This is the case that happens when we use $\mathtt{DegRevLex}$.
\end{proof}

For now on, we simply denote by $\prec$ the term ordering $\mathtt{DegRevLex}$ in $\FQ[x,y]$ with $y>x$.
Figure \ref{fig.base} represents the basis  for $A_E$ given in \eqref{baseB},  that in the following we will denote by $\B$.
\begin{figure}[ht!]
 \centering
\includegraphics[height=5cm]{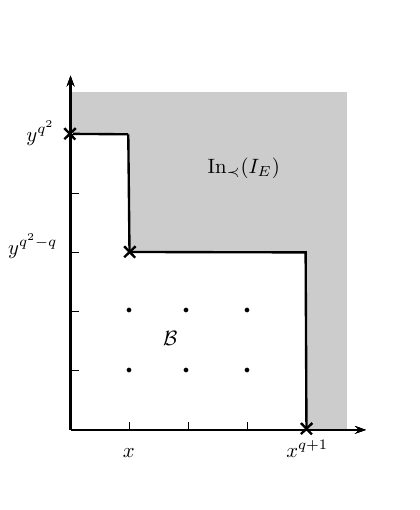}
  \caption{The set of monomials in $\mathcal B$}
  \label{fig.base}
\end{figure}

\begin{remark} \label{gaps} 
The monomial basis $\B$ for   the $\FQ$-vector space  $A_E$ allows us to represent every class in $A_E$ by the unique polynomial $F$ in this class such that  its  degree $\partial_x(F)$ with respect to  $x$ is at most $q$ and its  degree $\partial_y(F)$  with respect to  $y$ is at most  $q^2-1$. 
We observe that  $\prec$ and $<_w$ choose the same leading monomial for each  polynomial $F$ in  $ \FQ [x,y]$ such that $\partial_x(F) \leq q$.  

In fact, let us consider two monomials $x^{\alpha_1}y^{\beta_1}$ and $x^{\alpha_2}y^{\beta_2}$ in $\B$ such that $x^{\alpha_1}y^{\beta_1}\prec x^{\alpha_2}y^{\beta_2}$. This happens when either
$\alpha_1+\beta_1=\alpha_2+ \beta_2$ and $\alpha_1>\alpha_2$ or $\alpha_1+\beta_1<\alpha_2+ \beta_2$. 

 In the first case we have equality between positive integers $\alpha_1-\alpha_2 = \beta_2-\beta_1$ and so $(\alpha_1-\alpha_2)q < (\beta_2-\beta_1)(q+1)$, hence  $x^{\alpha_1}y^{\beta_1} <_w x^{\alpha_2}y^{\beta_2}$.\\
In the second case, we have $\alpha_1-\alpha_2 < \beta_2-\beta_1$.  If $\beta_2-\beta_1\geq 0$, then  $(\alpha_1-\alpha_2)q < (\beta_2-\beta_1)(q+1)$ hence $x^{\alpha_1}y^{\beta_1} <_w x^{\alpha_2}y^{\beta_2}$. It remains to consider the case $\beta_1=\beta_2+t$ and $\alpha_2=\alpha_1+t+k$, with suitable positive integers $t,k$. Then $(\alpha_1-\alpha_2)q - (\beta_2-\beta_1)(q+1) =-kq+t$ which is negative since $t< \alpha_2\leq q$.\\
  As a consequence we see that  the $w$-degrees of the monomials in $\B$  are pairwise different. Moreover, they cover the range of  integers  between $0$ and $q^3+(q^2-q)-1=n+2g-1$, with  $2g$  exceptions. In fact there are $g$ integers  between     $1$ and $2g-1$ that cannot be obtained as the $w$-degree of any monomial: they are the  gaps of the semigroup generated by  $q$ and $q+1$. Moreover, there are $g$ numbers between   $ n$ and $n+2g-1$ that are the $w$-degree of monomials multiple of $xy^{q^2-q}$ that do not belong to $\B$ since       $xy^{q^2-q}\in \In_\prec(I_E)$. 
 \end{remark}

Finally, we summarize  the explicit  construction of the  \textit{Hermitian code}.

\begin{definition}\label{def:codice}
Let us fix any  ordered list  $[P_1, \dots, P_n]$ of the  $n$ points of   $E$.  Let $\B$ be as \eqref{baseB},   $m$ be any integer  $  \leq n+2g-2$, and let  $V_m \subseteq A_E$ be the vector space  with basis 
$$
\B_m=\B \cap V_m=\{x^ry^s \in\B \mid rq+s(q+1)\leq m\}.
$$  

Then the  \textbf{Hermitian code} $C_m$  is the $\FQ$-vector space
$(\mathrm{Span}_{\FQ}\langle\phi_{I_E}(\B_m)\rangle)^\perp$ where $\phi_{I_E}$ is the evaluation map (\ref{eval}) at the points of $E$. 
\end{definition}

The Hermitian codes can be divided in four phases \cite{CGC-cd-book-AG_HB},
any of them having specific explicit formulas linking their dimension and
their distance \cite{CGC-cd-phdthesis-marcolla}, as in Table~\ref{Tab}. 

{\scriptsize{
\begin{table}[H]
\caption{The four \textquotedblleft phases\textquotedblright of Hermitian codes \cite{CGC-cd-phdthesis-marcolla}.}\label{Tab}
\begin{center}
\begin{tabular}[h]{|cccc|}
\hline\rowcolor{lightgray}\noalign{\smallskip}
\textsf{Phase}  & $\textsf{m}$ & \textsf{Distance} $\textsf{d}$ & \textsf{Dimension} $\textsf{k}$\\
\noalign{\smallskip}
\hline
\noalign{\smallskip}
\textbf{1} & {\scriptsize$\begin{array}{c}
  0\le m\le q^2-2\\
   m=aq+b\\
  0\le b\le a \le q-1\\
b\ne q-1
\end{array}$} & {\scriptsize$\begin{array}{ll}
  a+1 & a>b\\
  a+2 & a=b
\end{array}\iff d\le q$} & {\scriptsize$
  q^3-\frac{a(a+1)}{2}-(b+1)$}\\
  & & &\\
\textbf{2} &  {\scriptsize$\begin{array}{c}
  q^2-1\le m\le 4g-3\\
  m=(2g-2)+ aq-b\\
  2\le a\le q-1\\
  0\le b\le q-2\\
  \mbox{ if } a=q-1 \mbox{ then } b\geq 1
\end{array}$} & {\scriptsize$\begin{array}{ll}
  aq-b & \mbox{ if } aq-b\in\langle q,q+1\rangle\\
  aq & \mbox{ if } aq-b\not\in\langle q,q+1\rangle
\end{array}$} & {\scriptsize$
   n-m+g-1$}\\
  & & &\\
\textbf{3} & {\scriptsize$4g-2\le m \le n-2$} & {\scriptsize$m-2g+2$}& {\scriptsize$n-m+g-1$}\\
  & & &\\
\textbf{4} & {\scriptsize$\begin{array}{c}
  n-1\le m \le n+2g-2\\
  m=n+2g-2-aq-b\\
0\le b\le a\le q-2,
\end{array}$} & {\scriptsize$n-aq-b$}& {\scriptsize$\frac{a(a+1)}{2}+b+1$}\\
\hline
\end{tabular}
\end{center}
\end{table}}}

\section{Hermitian codes $C_m$ with $m \geq q^2 -1$.}\label{Sec.prere}

In the following,  for every Hermitian code  $C_m$ with $m\geq q^2 -1$, we will denote by $(\mu, \beta)$ the unique pair of non negative integers such that  $m=\mu q +\beta (q+1)$ and  $\mu \leq q$ (Remark \ref{gaps}).  As we are considering only the case $m\geq q^2 -1$, these integers do exist.
 
It is a straightforward consequence of Definition \ref{def:codice} that   $C_{m}\supseteq C_{m'}$ when  $m<m'$. If in $\B$ there is no monomial of $w$-degree   $m+1$,  then $\mathcal B_{m}=\mathcal B_{m+1}$ and there is only one code corresponding  to $m$ and to $m+1$.  
For this reason we adopt the following
\begin{convention}\label{conv.m} We will label a   code after an integer $m$ if and only if    $\B$ contains a monomial of $w$-degree  $m+1$.
\end{convention}
Therefore, by  Remark \ref{gaps},   we do not label a code after $m$   when either $m+1$ is a gap of the semigroup $\langle q,q+1\rangle$   or  the only monomial $x^ry^s$ with $r\leq q$ and  $w$-degree $m+1$ is a multiple of  $xy^{q^2-q}$.

 \begin{example}\label{ex32} For   $q=3$, the first integer  $m$  in  the rang of  the  fourth phase is    $26$.  However, we do not label any   code after this value of $m$.  In fact,   the only     monomial  $x^r y^s$ with $r\leq 3$ and  $w$-degree  $26+1$,  is  $xy^6$,   a minimal generator of $\In_\prec(I_E)=\langle x^4, xy^6, y^9\rangle$. Hence,  $\mathcal B_{26}=\mathcal B_{27}$. On the other hand, in $\mathcal B$ there is the monomial  $y^7$ whose $w$-degree is $27+1$, so that  $\mathcal B_{27}\neq \mathcal B_{28}$.  Therefore, the first code of the fourth phase will be denoted in our convention as $C_{27}$.
\end{example}

Note that  for every  monomial $x^r y^s$ of $\B$ we have  
$$ x^r y^s\in \B\setminus \B_m\Longleftrightarrow  \mbox{ either }    r+s >\mu + \beta  \hbox{ or } r+s=\mu + \beta   \hbox{ and }    r<\mu.$$

In the following, for any $\FQ$-divisor $D$ over the  Hermitian curve we will denote by $V_{m,D}$   the image  of $V_m$ in $A_D$.  
 We observe that
 $D$   contains  the support of some codeword of $C_m$ if $V_{m,D}$  has dimension less than $\vert  D\vert$.  
We   summarize all these facts in the following:

\begin{proposition}\label{fondamentale} 
Let $D$ be a $\FQ$-divisor over the Hermitian curve. Then the following are equivalent for the Hermitian code $C_m$ with $m\geq q^2-1$:
\begin{enumerate}
    \item[(i)] \label{fondamentale_i}  $\exists \ \mathbf c \in  C_m $ with $\mathbf c\ne 0$ such that $\supp(\mathbf c)\subseteq D$
\item[(ii)] \label{fondamentale_ii}   $V_{m,D}$   has dimension less than $\vert  D\vert$ as  $\FQ$-vector space.
\item[(iii)]  \label{fondamentale_iii} There exists a monomial $x^ay^b\in\cN(\In_\prec(I_D))$  whose $w$-degree $w(x^ay^b)=aq+b(q+1) $ satisfies $m+1\leq w(x^ay^b) \leq m+q+1$.
\item[(iv)] \label{fondamentale_iv}  Let $m=\mu q + \beta (q+1)$. If $\mu >0$, $\cN(\In_\prec (I_D))$ contains at least one monomial in 
\begin{equation}\label{scala}  \mathcal L_1:=\{ x^{\mu-i}y^{\beta+i} \hbox{ with } 1\leq i \leq \mu \} \cup \{ x^{q-j}y^{\beta+\mu -q+j+1} \hbox{ with } 0\leq j\leq q-\mu\};  \end{equation}
 otherwise,  $\mu=0$  and $\cN(\In_\prec (I_D))$ contains at least one monomial in
\begin{equation}\label{scala2}   
\mathcal L_2:=\{ x^{q-j}y^{\beta -q+j+1}, \hbox{ with } 0\leq j\leq q\}. 
\end{equation}
\end{enumerate}  
\end{proposition}
\begin{proof} 
Let $\rm H_E$ be  the parity-check matrix   of the repetition code (the only one with distance $q^3$,  whose codewords are the $q^2$ vectors with all the entries pairwise equal)  and   $\rm H_{E,m}$ be  the  parity-check matrix   of the code  $C_m$. Let us  consider their sub-matrices $M_D$ and $M_{D,m}$ obtained only considering the columns corresponding to the points $P_i \in D$.  More explicitly, if $\delta=\vert D\vert$ and $D=\{P_{i_1}, \dots, P_{i_\delta}\}$   the  $j$-th  column of $\rm H_E$  (respectively of $\rm H_{E,m}$)  is  formed by the  evaluation of the monomials of $\B$ (respectively of   $\B_m$)  at $P_{i_j}$. \\
The   codewords of  $C_m$ are the solutions of the homogeneous   linear system $\rm H_{E,m}\cdot  Z^T=0$, where $Z=(z_1, \dots, z_n)$. The list of non-zero components of any  codeword  $\bc\in C_m$ with  $\supp(\bc)\subseteq D$ is    a solution of  the linear system $M_{D,m}\cdot Z_D^T=0$, where $Z_D=(z_{i_1}, \dots, z_{i_\delta})$.  Therefore,  there exist non-zero codewords $\bc\in C_m$ such that $\supp(\bc)\subseteq D$ if and only if $\rk(M_{D,m})<\vert D\vert$, that is, \eqref{fondamentale_ii} is verified.\\
\noindent By Lemma \ref{eqcampo}   the   basis of $A_D$ given by the monomials of  $\cN(\In_\prec(I_D))$  is a subset of $\B$. Since $\rk(M_{D})=\vert D\vert$, then  \eqref{fondamentale_i} is verified if and only if there exists a monomial   $x^ry^s \in \B\setminus \B_m$  that belongs to  $\cN(\In_\prec(I_D))$.
We may assume that $x^ry^s$ is the monomial in $\cN(\In_\prec(I_D)) \cap (\B\setminus\B_m)$ having minimum $w$-degree.
 
This $w$-degree is less than or equal to $m+q+1$. In fact, if this were not true we could find in $\cN(\In_\prec(I_D))$ either  $x^{r-1}y^s$ or $x^ry^{s-1}$ with $w$-degree still larger than $m$ against the minimality of  $x^ry^s$. \\
The equivalence between \eqref{fondamentale_iii}  and \eqref{fondamentale_iv} is obvious, being \eqref{fondamentale_iv} a more explicit rewriting  of \eqref{fondamentale_iii}.
\end{proof}

Observe that the equivalence between the conditions   \eqref{fondamentale_i},\eqref{fondamentale_ii},\eqref{fondamentale_iii} holds true 
for every Hermitian code $C_m$ and also \eqref{fondamentale_iv} is equivalent to the previous ones provided the integer $m$ can be written as $\mu q+\beta(q+1)$.

\subsection{Complete intersections on $\He$}

In this section we use the B\'{e}zout Theorem to find some properties of  the zero-dimensional schemes that are complete intersection of  $\He$ with another curve $\xx$. 
To this purpose, we  must also  consider  the possible intersections at infinity.
We recall that the projective closure $\overline \He$ of $\He$ has a single  point at infinity  $P_\infty=[x_0=0:x_1=0:x_2]$ where $x_1/x_0=x$ and $ x_2/x_0=y$. It is a smooth, inflexion  point with  tangent line $x_0=0$.

\begin{proposition}\label{gradox} 
Let $F\in\FQ[x,y]$ be a polynomial such that $\partial_x (F)\leq q$ and let $\xx$ be the curve given by the ideal $\langle F\rangle $. If  $\LM_\prec (F)=x^r y^s$, then 
\begin{enumerate}
\item[(i)] \label{gradox_i} $\In_\prec(\langle H,F\rangle)=\langle x^{q+1}, x^ry^s, y^{s+q} \rangle$ when    $s>  0$ and 
\item[(ii)] \label{gradox_ii}  $\In_\prec(\langle H,F\rangle)=\langle x^{r},  y^{q} \rangle$ when $s=0$.
\end{enumerate}
Moreover,   the degree of the divisor $D$ cut on  $\He$ by  $\xx $ is $rq+s  (q+1)$. 
\end{proposition}
\begin{proof}
We first prove that $|D|=rq+s  (q+1)$.  Since the term ordering $\prec$ is degree-compatible, the degree  of $F$ is equal to    the degree $r+s$ of its leading monomial. By  B\'{e}zout Theorem, the degree of the divisor  $\overline D$ cut on $\overline \He$ by the  projective closure $\overline \xx$ of $\xx$   is  $(r+s)(q+1)$.  It remains to prove that    the intersection multiplicity   of  the two curves at $P_\infty$ is $r$. For the generalities about the intersection multiplicity of two plane curves we refer to \cite{fulton:curve}.

 To this aim, we study the intersection of the two curves looking at any open subset of $\mathbb P^2$ around $P_\infty$; for instance we  choose the affine chart  given by $x_2\neq 0$. 

 In order to obtain the equations defining the two curves in this affine chart   we homogenize   $F$ and $H$, by setting $\overline F:=F(x_1/x_0,  x_2/x_0)\cdot x_0^{r+s}$ and $\overline H:=H(x_1/x_0,  x_2/x_0)\cdot x_0^{q+1}=x_1^{q+1}-x_2^q x_0-x_2 x_0^q$. Then  we  de-homogenize by setting $x_2=1, x_0=z, x_1=x$ and get the wanted equations $F'$ and $H'=x^{q+1}-z^q-z$.  For what concerns $F'$,  we observe that all monomials  of maximum degree in the  support of $F$ are divisible by $x^r$; hence  every  monomial of   $F'$ is divisible by $z$ and/or by $x^r$. Furthermore, the monomial    $x^r$ appears in the support of  $F'$, since $x^r y^s$ is in the support of $F$ and we have performed the following transformations
 $$x^r y^s \mapsto \frac{x_1^r}{x_0^r}\cdot  \frac{x_2^s}{x_0^s}\cdot x_0^{r+s} =x_1^r x_2^s \mapsto x^r. $$

Without modifying the intersection multiplicity at $P_\infty$,    we can replace any occurrence  of $z$ in $F'$ by $H'+z=x^{q+1}-z^q$, and repeat this substitution until we get  a polynomial $F''$ having $x^r$ as the only monomial of  minimum degree.
Therefore, $mult_{P_\infty}(\overline \He, \overline \xx)=mult_{P_\infty}(H',F')=mult_{P_\infty}(H',F'')=r$.

This allows us to conclude that   $|D|=\vert \overline D \vert - \vert rP_\infty \vert=  rq+s  (q+1)$.  

\medskip

Now we prove  \eqref{gradox_i} and  \eqref{gradox_ii}.  As the set of monomials $\cN(\In_\prec (I_D) )$ is a basis for $A_D$, for what just proved we know that its cardinality is $rq+s(q+1)$.
\begin{itemize}
    \item[(i)] If $s>0$, we can write  $F$ as $x^r y^s + x^{r+1}F_1+F_2$ where $\partial F_1=s -1$ and $\partial F_2< r+s $. We observe that the polynomial  $x^{q+1-r}F-(y^s +xF_1) H$ is  an element of $I_D$ and its  leading monomial is $y^{s+q}$. Therefore,  $\In_\prec (I_D) \supseteq J:=\langle x^{q+1}, x^r y^{s}, y^{s+q}\rangle$, so that  $\cN(\In_\prec (I_D) )\subseteq \cN (J)$. It is now easy to check that  the cardinality of $\cN (J)$ is exactly  $rq+s(q+1)$ and get the equality  $\In_\prec (I_D)= J$. 
    \item[(ii)] If $s=0$  we can write  $F$ as $x^r+F_3$ where $\partial F_3<r$. Again, we see that the  polynomial  $H-x^{q+1-r}F$ belongs to  $I_D$ and its leading monomial   is $y^q$. Hence  $\In_\prec (I_D) \supseteq J':=\langle x^{r},   y^{q}\rangle$, so that   $\cN(\In_\prec (I_D) )\subseteq \cN (J')$. An easy computation shows that   $\vert \cN (J')\vert = rq$ and we conclude  that $\In_\prec (I_D)=J'$.
\end{itemize}
\end{proof}

\begin{corollary}\label{intersezioniPinf} 
Let $D$ be a divisor over $\mathcal H$ and let $x^r y^s$  be a monomial in $\In_\prec (I_D)$ with  $r\leq q$. Then $\vert D\vert \leq  r q + s (q+1)$.
\end{corollary}
\begin{proof}
Let $F$ be any polynomial in $I_D$ such that $\LM_\prec (F)=x^r y^s$.  Then  the degree of  $F$ is  $r+s$ and  the projective closure of the curve defined by   $F$ cuts on $\overline\He$  a divisor $D+D'+t P_\infty$ of degree $ (r+s)(q+1)$.  By Proposition~\ref{gradox} we know that  $t=r$ and so $\vert D\vert \leq \vert D+D'\vert=(r+s)(q+1) -r= r q + s (q+1)$. 
\end{proof}

\begin{lemma}  \label{unodeidue} 
Let  $D$ be a $\FQ$-divisor over $\He$ which is the support of a non-zero  codeword of $C_m$. Then $I_D$ verifies at least one of the following conditions:
\begin{enumerate}
\item[(i)] \label{unodeidue_i}   $x^{q+1}\in \In_\prec (I_D)$  and $x^{q} \notin \In_\prec (I_D)$
\item[(ii)] \label{unodeidue_ii}  $y^q \in \In_\prec (I_D)$.
\end{enumerate} 
If moreover  $m\geq  2q^2-2q-2$, then either  $x^{q} \notin \In_\prec (I_D)$, or $y^{q-1} \notin \In_\prec (I_D)$.
\end{lemma}
\begin{proof}
Since  $I_D$ is a zero-dimensional ideal, then $ \In_\prec (I_D)$ contains some power of $x$ and of $y$: let  $x^r$ and $y^s$ be the minimal ones.
Obviously,  $r\leq q+1$, as $H\in I_D$. We  assume $r\leq q$, and prove that  $s\leq q$.

Indeed, if $F$ is any polynomial in $ I_D$   with leading monomial $x^r$, then we find in $I_D$ also the polynomial  $H-x^{q+1-r}F$ whose leading monomial is $y^q$. \\
It remains to prove that when $m\geq  2q^2-2q-2$ we cannot have both $r \leq q$ and $s \leq q-1$. In fact, if so,  $\In_\prec (I_D) \supseteq \langle x^q, y^{q-1}\rangle$, that is  $\cN(\In_\prec (I_D) )\subseteq \cN(\langle x^q, y^{q-1}\rangle)$.  The larger   $w$-degree of monomials in  $\cN((x^q, y^{q-1})) $ is $w(x^{q-1}y^{q-2})=2q^2-2q-2\leq m$,   in contradiction with  Proposition~\ref{fondamentale}. 
\end{proof} 

\section{Minimum distance of Hermitian codes of  third and fourth phase}\label{Sec.dist}

In this section we study the Hermitian codes $C_m$ with $ m\geq 2q^2-2q-2$ and in particular, at the end of the section, we  get a formula for their distance. What we obtain is nothing else than  the well known formula first proved by Stichtenoth in \cite{stichtenoth1988note}. However, we prefer to prove it directly since the preliminary  results that will lead us to this proof, especially  Theorem \ref{diseguaglianza},   are key tools in the main results of this paper.

\medskip

\begin{remark} \label{betamaggiore}
 We observe that in the new hypothesis, the numbers $\mu$ and $\beta$ such that $m=\mu q + \beta(q+1)$  always satisfy the  inequalities   $0\leq \mu \leq q$ and  $\beta \geq q-2$. More generally if $a,b$ are non-negative integers such that $ a\leq q$ and  $aq+ b (q+1)\geq  2q^2-2q-2$, then $b \geq q-2$. 
\end{remark}

  The argument we use to prove the last  part of the following  theorem is directly inspired by  the one used  by Stichtenoth in \cite{stichtenoth1988note} to explicitly exhibit  some codewords  of minimum weight and exploit the norm-trace form of the equation of $\He$.

\begin{theorem} \label{diseguaglianza} 
Let $x^a y^b\in \B\setminus \B_m$ such that   $m+1\leq w(x^a y^b)\leq m+q+1$.  
\begin{enumerate}
\item[(i)] If  $D$ is a $\FQ$-divisor over $\mathcal H$  such that  $x^a y^b\in \cN(\In_\prec (I_D))$ and which  is the support of a codeword of $C_m$, then
 $\vert D \vert \geq w(x^a y^b)- (q^2-q)+1$.
\item[(ii)]  There exists a $\FQ$-divisor $D'$ over $\He$ corresponding to a codeword of $C_m$ such that  $x^a y^b\in \cN(\In_\prec (I_{D'}))$ and $\vert D' \vert = w(x^a y^b)- (q^2-q)+1$. 
\end{enumerate}
\end{theorem}
\begin{proof}   
By Remark \ref{betamaggiore} we have that $0 \leq a \leq q$ and $b\geq q-2$. We split the proof of the first item into  three  cases:
\begin{itemize}
\item  If $b=q-2$, then $w(x^ay^b)\geq m+1$ if and only if $a=q$. 
Since  $x^a y^b=x^q y^{q-2}\in\cN(\In_\prec(I_D))$, then $\cN(\In_\prec(I_D))$ also contains all $(q+1)(q-1)$ factors of  $x^q y^{q-2}$. Therefore,      $\vert D\vert= \vert \cN(\In_\prec(I_D))\vert  \geq (q+1)(q-1)=w(x^a y^b)- (q^2-q)+1$. 
\item If  $b=q-1$, we can argue as in the previous case:   from $x^a y^{q-1}\in \cN(\In_\prec(I_D))$ we get    $\vert D\vert= \vert \cN(\In_\prec(I_D))\vert  \geq (a+1)q=w(x^a y^b)- (q^2-q)+1$. 
\item If  $b\geq q$, then    $y^q\notin \In_\prec (I_D)$. As a consequence, we know by Lemma~\ref{unodeidue}   that $x^{q+1}$ is the minimal power of $x$ in $\In_\prec(I_D)$ and we deduce    by Proposition \ref{gradox}  that  $x^q y^{b-q}\not\in\In_\prec(I_D)$.

\begin{tabular}{ccc}
\begin{minipage}{7.5cm}
 Computing the number of factors of the  two monomials   $x^a y^b$ and $x^q y^{b-q}$ we get  $\vert D\vert= \vert \cN(\In_\prec(I_D))\vert  \geq (a+1)(b+1)+(q+1)(b-q+1)-(a+1)(b-q+1)=w(x^a y^b)- (q^2-q)+1$.
\end{minipage}

&

&
\begin{minipage}{4cm}

\includegraphics[width=4cm]{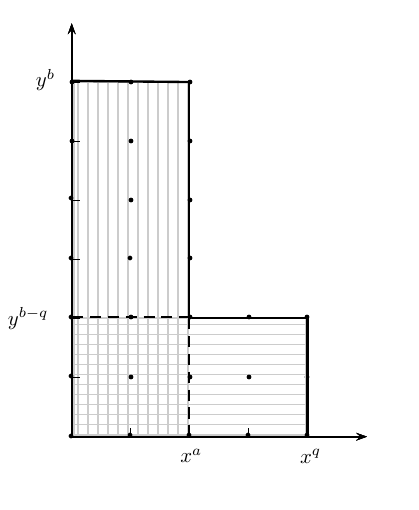}
\end{minipage}

\end{tabular}
\end{itemize}

\noindent Now we prove the second item, again splitting  in cases. In all of them we can obtain  $D'$ as the divisor cut on  $\He$ by  a  curve $\y$ union of lines  chosen among those cutting    on $\He$ an  $\FQ$-divisor. Recall that there are  $q^2$ vertical lines  $\mathcal L_i$, each containing $q$ points of $E$, and $q^2-q+1$ horizontal lines $\mathcal L'_j$,  each containing $q+1$ points of $E$   (Remark~\ref{griglia}).

\begin{itemize}
\item  If  $a=q$ and $b\in\{ q-2,q-1\}$,   $\y$ is the union of $b+1$ horizontal lines  $\mathcal L'_j$, 
 so  that $\In_\prec(I_{D'})=\langle x^{q+1},y^{b+1}\rangle$.

\item  If  $a=b=q-1$,   $\y$ is the union of $q$ vertical lines  $\mathcal L_i$,   
so that  $\In_\prec(I_{D'})=\langle x^{q},y^q\rangle$.

\item If $b\geq q$,  
 $\y$ is the union of $a+1$ vertical lines $\mathcal L_i$ and $b+1-q$ horizontal lines $\mathcal L_j'$   such that not two of them meet in a point of $E$. 
\\
For instance we can choose among the $\mathcal L_i$  a set of  lines   $x$, $x-\alpha_1, \dots, x-\alpha_a$ with $\No(\alpha_i):=\alpha_i^q= 1$ and among the $\mathcal L_j'$  a set of  lines  $y-\beta_1, \dots, y-\beta_{b+1-q}$    with $\Tra(\beta_j):=\beta_j^{q+1}-\beta_j\neq 0,1$.
 In this case, we have $\In_\prec(I_{D'})=\langle x^{q+1},y^{b+1},x^{a+1}y^{b+1-q}\rangle$.
\end{itemize}
\end{proof}
 We can now prove a general formula for the distance of the codes of the third and fourth phase.

\begin{corollary}\label{cor.dist}  The distance of an     Hermitian code $C_m$ with  $m\geq 2q^2-2q-2$   is $d=m-q^2 +q+2=m -2g+2$.
\end{corollary}
\begin{proof} By Proposition \ref{fondamentale}, if $D$  is a $\FQ$-divisor  over $\He$ corresponding to a codeword of $C_m$, then  $ \cN(\In_\prec (I_D))$ contains  a  monomial $x^a y^b$ of $\B\setminus\B_m$. 
By  Theorem~\ref{diseguaglianza}, we have  $\vert D\vert \geq w(x^a y^b)-q^2 +q+1$,  and there is a $\FQ$-divisor $D'$ over $\He$ for which we have  equality.  Therefore, the minimum $w$-degree of such  divisors, namely the distance $d$ of the code $C_m$, is that obtained when $x^a y^b$ is the monomial of $w$-degree in $\B\setminus\B_m$, that in our assumption is always $m+1$.  
\end{proof}

Note that the formula  given in the previous corollary  coincides with that of \cite{CGC-cd-book-AG_HB}, taking in account Convention~\ref{conv.m}  on the integers $m$ labeling a code. Let for  instance $C$ be  the first code   of the fourth phase (with $q$=3)  as in Example~\ref{ex32}. It corresponds to     $\mathcal B_{26}=\mathcal B_{27}$ so that it could be labeled after both values  $m=26$ and $m=27$.  We choose to label it as $C_{27}$ and the formula in Corollary \ref{cor.dist} gives  for $m=27$ the correct distance $d=23$ (while $m=26$ would not). 


\section{Geometric description of minimum weight codewords}\label{Sec.MinWord}
 
In this section we consider Hermitian codes
$C_m$  with  $2q^2-2q-2\leq m \leq n+2g-2$.

   \noindent  By Remark~\ref{betamaggiore} and Corollary~\ref{cor.dist},  there is a unique writing  of $m$ as $ \mu q+\beta(q+1)$ with  $0\leq \mu \leq q$,  $\beta \geq q-2$,  and  the distance of $C_m$ is $d=\mu q+\lambda(q+1)$ where $\lambda=\beta -q+2$. Moreover,   $\B$ contains  only one  monomial $x^a y^b$ with  $w$-degree $m+1$.\\

\begin{theorem}\label{curvagenerica} 
Let $C_m$ be an    Hermitian code, and let $m$, $\mu$, $\beta$, $\lambda$, $a$, $b$ be as above.    
\begin{enumerate}
\item[(i)] \label{curvagenerica_i} Let $D$ be a $\FQ$-divisor  cut  on  $\He$ by the curve defined by a polynomial $F$  such that $\LM_\prec (F)=x^\mu y^{\lambda}$. Then $D$  corresponds to  minimum weight codewords of $C_m$.
\item[(ii)] \label{curvagenerica_ii}
Let $D$  be a $\FQ$-divisor  on $\mathcal H$ corresponding to  minimum weight codewords of $C_m$.  Then $\cN(\In_\prec  (I_D))$ contains   only one monomial  with   $w$-degree   larger  than $m$, the monomial   $x^a y^b$,    with   $w$-degree   exactly equal to $m+1$.  Moreover $D$ is  cut  on  $\He$ by the curve defined by a polynomial $F$    such that   $\LM_\prec (F)=x^\mu y^{\lambda}$.
\end{enumerate}
\end{theorem}

\begin{proof} 
\begin{enumerate}
    \item[(i)]  By Proposition \ref{gradox},  if   $D$ is a $\FQ$-divisor  cut  on  $\He$ by the curve defined by a polynomial $F$  such that  $\LM_\prec (F)=x^\mu y^{\lambda}$, then $|D|=\mu q+\lambda (q+1)$. Moreover,  if $\lambda\geq 1$,  then $\LM_\prec (\langle H,F \rangle )=\langle x^{q+1},x^\mu y^\lambda,y^{\lambda+q}\rangle $;  otherwise  $\lambda=0$  and  $\In_\prec (\langle H,F \rangle)=\langle x^\mu,y^q\rangle$.
In both cases, $\cN(\In_\prec (I_D))$ contains   a monomial of $w$-degree $m+1$. By Proposition \ref{fondamentale},  $D$ corresponds to codewords of weight $|D|=d$.
    \item[(ii)]   It is easy to see that   $x^a y^b=x^{\mu-1}y^{\lambda+q-1}$  if $\mu\geq 1$, and  $x^a y^b=x^q y^{\lambda-1}$ if $\mu=0$. Moreover, by Theorem \ref{diseguaglianza}, among the $d$ monomials of $ \cN(\In_\prec (I_D))$ we find $x^a y^b $ and all its factors: then  $d\geq \mu (\lambda+q)$ if $\mu \geq 1$ and $d\geq (q+1)\lambda$ if $\mu=0$.  
\begin{center}
\begin{figure}[ht!]
  \hspace{2cm}
\includegraphics[height=4cm]{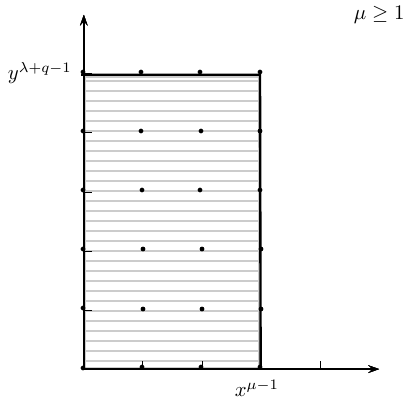}
  \hspace{2cm}
\includegraphics[height=4cm]{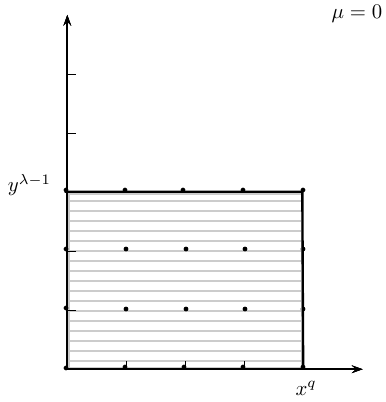}
  \caption{The divisors of $x^{\mu-1}y^{\lambda+q-1}$ when $\mu\geq 1$ and of $x^q y^{\lambda-1}$ if $\mu=0$.}
  \label{fig.div1}
\end{figure}
\end{center}
We consider tree different cases:
\begin{itemize}
\item Case $\mu=0$.  Since $d=\lambda(q+1)$,  the number  of  monomials dividing  $x^a y^b$ are exactly  as many as the monomials in  $\cN(\In_\prec (I_D))$,  hence   $\cN(\In_\prec (I_D))$ is exactly the set of these monomials.  Therefore,  there exist a polynomial $F$ in $I_D$ such that $LM_\prec(F)=y^\lambda $. By Corollary \ref{intersezioniPinf} this polynomial  cuts on $\He$ a divisor $D'$ which contains $D$   and  has the same degree of $D$, so $D'=D$.

\item Case $\mu,\lambda >0$. We observe that  $x^q y^{\lambda-1}$ must belong to $\cN(\In_\prec (I_D))$. Otherwise,  by Proposition~\ref{gradox}, we would $y^{\lambda+q-1}\in \In(I_D)$ and so also    $x^{\mu-1}y^{\lambda+q-1}\in \In_\prec (I_D)$, against   what  we have just  proved.   
\begin{figure}[ht!]
  \centering
\includegraphics[height=5cm]{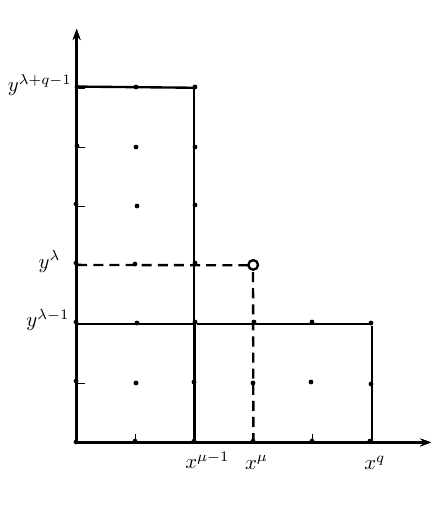}
  \caption{The number of monomials in   $\cN(\In_\prec (\langle H, F\rangle ))$.}
  \label{fig.div3}
 \end{figure}
Computing the number of the  divisors of  $x^{\mu-1}y^{\lambda+q-1}$  and of  $x^q y^{\lambda-1}$, we get  exactly $d$ (see Figure \ref{fig.div3}). Therefore, $ \cN(\In_\prec (I_D))$ is formed by these monomials, since $ \vert \cN(\In_\prec (I_D))\vert =d$. Moreover $x^\mu y^\lambda$, that does not divide neither $x^{\mu-1}y^{\lambda+q-1}$ nor $x^q y^{\lambda-1}$,  belongs to $\In_\prec (I_D)$. Then, there is a polynomial $F\in I_D$ with leading monomial  $x^\mu y^\lambda$. By Corollary \ref{intersezioniPinf}  the curve defined by  $F$ cuts $D$ on $\He$.

\item Case $\lambda=0$.   The monomial $x^a y^b=x^{\mu-1}y^{q-1}$    has exactly   $\mu q=d$ divisors. So  $\cN(\In_\prec (I_D))$ is formed by  these monomials.  Therefore $x^\mu $ is an element of  $\In_\prec (I_D))$. If $F\in I_D $ has leading monomial $x^\mu$, by Proposition  \ref{gradox},   $I_D$ contains   the ideal generated by $H$ and $F$ whose initial ideal is $( x^\mu, y^q)$.  So $I_D=( H,F )$ and the curve defined by  $F$ cuts $D$ on $\He$.
\end{itemize}
\end{enumerate}
\end{proof}

\begin{remark} 
Let $C_m$ be the Hermitian code with distance $d=\mu q+\lambda(q+1)$. As shown in the proof of Theorem \ref{diseguaglianza}, 
we obtain some  divisor $D$ corresponding to a minimum-weight codeword of $C_m$ cutting $\mathcal H$ with a suitable curve $\xx$ union  of  $\mu$ vertical lines and $\lambda$ horizontal  lines, so that 
the leading monomial of the polynomial defining such a curve $\xx$ is indeed $x^\mu y^\lambda$.\\
On the other hand, we point out that not  all the polynomials $F$ in $\FQ[x,y]$ with leading monomial  $x^\mu y^{\lambda} $ correspond to  minimum-weight codewords of the code $C_m$. The following result contains an explicit characterization of the \lq\lq good\rq\rq polynomials  as a consequence of  Remark \ref{rm:eqcampo}. 
\end{remark}

\begin{corollary}\label{corollfin}
Let $C_m$ be the Hermitian code with  distance $d=\mu q+\lambda(q+1)$. 
A polynomial $F$ over $\FQ$ with $\LM_\prec (F)=x^\mu y^{\lambda}$ cuts over $\He$ a divisor $D$ corresponding to  minimum-weight codewords if and only if  the ideal $\langle H,F\rangle$ contains the field equations $x^{q^2}-x$ and $y^{q^2}-y$.
\end{corollary}


\section{Algorithm to compute the number of minimum-weight codewords}\label{Sec.alg}

Let $C_m$ be an Hermitian code of either third or fourth phase. Based on the results obtained in the previous sections, and mainly on Theorem \ref{curvagenerica}, we now describe   an algorithm computing the number  $\texttt{MW}_m$   of  minimum-weight codewords of $C_m$ can be built.    Then,    we propose Algorithm~\ref{Alg} which takes as inputs
 $q,m$, and returns  $\texttt{MW}_m$.

We know  that  the integer $m$ such that  $C_m$ is a  code in ether  the third or the  fourth phase  can be written as   $m=\mu q+ (\lambda+q-2  )(q+1)$ with $\mu$, $\lambda$ non-negative integers and  $\mu\leq q$; moreover we know a formula giving   the distance of $C_m$ as a function of  $\lambda$ and $\mu$.   \ecr We identify the code $C_m$ with the triple of integers $(q, \lambda, \mu)$. We recall that $\prec$ denotes the term ordering $\mathtt{DegRevLex}$  in $\FQ[x,y]$ with $y>x$.\ecr 

\begin{itemize}
\item  We compute  the distance $d=\mu q+ \lambda(q+1)$ of the code $C_m$. 
\item We compute the set $M$ of the pairs $(a,b)$ of exponents of the  monomials $x^a y^b$  such that $x^a y^b \prec x^\mu y^{\lambda}$ and $a\leq q$. 
\item Now we exploit Theorem \ref{curvagenerica}. For every  monomial $(a,b)$ in $M$, we introduce a  variable  $\nu_{a,b}$. Then we set $$f:=x^\mu y^{\lambda}+\sum_{(a,b)\in M} \nu_{a,b}  x^a y^b.$$
  (line \ref{step.poly} of Algorithm~\ref{Alg}). 

\item We compute the Gr\"obner  basis $\mathcal G(I)$  of the ideal $I=\langle f,H\rangle$ in $\FQ[T, x,y]$,   where  $T=\{ \nu_{a,b} \}_{(a,b)\in M}$, w.r.t a block term ordering that is  $\mathtt{DegRevLex}$ with $y>x$ for the first block $\{y,x\}$  and is any term ordering on  the second block $T$.   Note that, by Proposition \ref{gradox}  we obtain a Gr\"obner basis $\mathcal G(I)$  formed by two or three polynomials whose leading monomials only contains the variables $x,y$    (line \ref{step.GB} of Algorithm~\ref{Alg}).

 \item We compute the normal forms $N_x$ of $x^{q^2}-x$ and $N_y$ of $y^{q^2}-y$ with respect to $\mathcal G(I)$  (lines \ref{step.NF1}-\ref{step.NF2} of  Algorithm~\ref{Alg}).     By Remark \ref{rm:eqcampo}, a  specialization  of the coefficients $\nu_{a,b}$ in $\FQ$  corresponds to a curve that cuts on $\He$ an $\FQ$-divisor if and only if under  this specialization  the ideal $I$ contains the polynomials  $x^{q^2}-x$ and $y^{q^2}-y$, hence if and only if the  normal forms $N_x$ and $N_y$  vanish.

\item We collect the two normal forms w.r.t.  the variables $x,y$ and coefficients  in $\FQ[T]$ and construct  the ideal $J$ in $\FQ[T]  $  generated by the coefficients of the two normal forms and by the field equations $\nu_{a,b}^{q^2}-\nu_{a,b}$ for every $\nu_{a,b}\in T$ (line \ref{step.J} of Algorithm~\ref{Alg}).

\item  By Remark \ref{rm:eqcampo},   $J$ is the reduced ideal   of a finite set of points $\mathcal V(J)$ in $\mathbb{A}_{\FQ}^N$,  where   $N=\vert T\vert$.   We compute the cardinality of $\mathcal V(J)$ computing either all its elements or the constant $z$ which is the   affine Hilbert polynomial  of $\FQ[T]/J$   (line  \ref{step.HP} of Algorithm~\ref{Alg}). 

 \item The set of points in  $\mathcal V(J)$   corresponds to the set of  minimum-weight codewords of $C_m$. More precisely,  to each point in $\mathcal V(J)$ we associate an homogeneous  linear system   whose solutions form  a  $1$-dimensional $\FQ$-vector space and  (except the zero one)  are  minimum-weight codewords of $C_m$.
Therefore,  the number     $\texttt{MW}_m$   of minimum-weight codewords of $C_m$ is $\vert \mathcal V(J)\vert \cdot (q^2-1)=z\cdot  (q^2-1)$  (line  \ref{step.num} of Algorithm~\ref{Alg}). 
\end{itemize}

Let us suppose that the following functions are available (they are present in the libraries for most computer algebra softwares as for instance MAGMA):
\begin{itemize}
\item[-] \texttt{GroebnerBasis}$(I)$ computing the Gr\"obner  basis $\mathcal G(I)$  of the ideal $I$ w.r.t. the term ordering $\mathtt{DegRevLex}$.
\item[-]  \texttt{NormalForm}$(f,G)$ computing the normal form of $f$ with respect to the Gr\"obner  basis $G$.

\item[-]   \texttt{HilbertPolynomial}$(I)$ computing the affine  Hilbert polynomial\begin{footnote}{In our implementation we  obtain the affine Hilbert polynomial of the ideal $I$ exploiting the MAGMA commands that compute the initial ideal $\In(I)$ with respect $\mathtt{DegRevLex}$ and  then the homogeneous Hilbert polynomial of $\In(I)$ with an additional variable.}
\end{footnote}  of a
zero-dimen\-sional ideal $I$.
\end{itemize}

\begin{algorithm}
	\caption{Compute the number of minimum-weight codewords}
	\label{Alg}
	\begin{algorithmic}[1]
		\Function{MinimumWords}{$ q,\lambda, \mu$}
		
				\State $d \gets \mu q + \lambda(q+1)$ 
				\State $M \gets \{(a,b) \mid a,b\in \NN,   a\leq q, \,a q + b(q+1)< d \}$  
		\State \label{step.poly}  $f \gets x^\mu y^\lambda+ \sum_{(a,b) \in M} \nu_{a,b}x^{a} y^{b} $ 
		\State $I \gets \langle H, f\rangle$ 
		\State \label{step.GB} $G\gets \mathtt{GroebnerBasis}(I)$ 
		\State \label{step.NF1} $N_x=\sum_{k,j}\beta_{k,j}x^ky^j\gets \mathtt{NormalForm}(x^{q^2}-x,G)$
		\State \label{step.NF2} $N_y=\sum_{\iota,\kappa}\gamma_{\iota,\kappa}x^\iota y^\kappa \gets \mathtt{NormalForm}(y^{q^2}-y,G)$ 
		\State \label{step.J} $J \gets \langle \{\beta_{k,j}\}_{k,j}, \{\gamma_{\iota,\kappa}\}_{\iota, \kappa}, \{ \nu_{a,b}^{q^2}-\nu_{a,b}\}_{(a,b)\in M}\rangle$ 
		\State  \label{step.HP} $z\gets \mathtt{HilbertPolynomial}(J)$ 
		\State \label{step.num}  $  \texttt{MW}_m   \gets z\cdot (q^2-1)$ 
		\State
		\Return   $\texttt{MW}_m$.
		\EndFunction
	\end{algorithmic}
\end{algorithm}

We have implemented this algorithm using the MAGMA software.  In MAGMA is already present a function  \texttt{MinimumWords} computing  the number of  minimum-weight codewords. We show in Table~\ref{Tab.q3.3fase} and in Table~\ref{Tab.q4.3fase}  the   time needed to compute this number   for   some codes of the third and fourth phase   using  the two  algorithms. 
For the  computations we used  8 Intel(R) Xeon(R) CPU  X5460  $@$ 3.16GHz, 4 core, 32GB of RAM with  MAGMA version V2.22-5.

\begin{table}[h]
\caption{Comparison time to compute the number of minimum-weight codewords of $C_m$ in the third ($10\leq m\leq 25$) and in fourth phase ($m=27,28$)  with $q=3$, using Algorithm \ref{Alg} and MAGMA command \texttt{MinimumWords}.} \label{Tab.q3.3fase} 
\centering 
\begin{tabular}{|c|c|c|c|c|}
\hline
$m$ & $d$ &   $\texttt{MW}_m$ & Algorithm \ref{Alg} & \texttt{MinimumWords}\\
\hline\hline
$10$ & $6$ & $576$ & $\mathbf{0.000}$ \textbf s & $15.100$ s \\
\hline
$11$ & $7$ & $2160$ & $\mathbf{0.020}$ \textbf s & $163.420$ s \\
\hline
$12$ & $8$ & $5400$ & $\mathbf{0.250}$ \textbf s & $191.560$ s \\
\hline
$13$ & $9$ & $8448$ & $\mathbf{2.120}$ \textbf  s & $68.100$ s \\
\hline
$14$ & $10$ & $17280$ & $9.920$ s & $\mathbf{7.040}$ \textbf  s \\
\hline
$15$ & $11$ & $24408$ & $119.330$ s & $\mathbf{2.440}$ \textbf  s \\
\hline
$16$ & $12$ & $32544$ & $467.680$ s & $\mathbf{1.450}$ \textbf  s \\
\hline
$17$ & $13$ & $39744$ & $1478.540$ s & $\mathbf{1.620}$ \textbf s \\
\hline
$18$ & $14$ & $39744$ & $1892.300$ s & $\mathbf{1.230}$ \textbf  s \\
\hline
$19$ & $15$ & $32544$ & $2073.710$ s  & $\mathbf{0.300}$ \textbf  s \\
\hline
$20$ & $16$ & $24408$ & $3052.910$ s  & $\mathbf{0.050}$ \textbf  s \\
\hline
$21$ & $17$ & $17280$ & $1314.380$ s  & $\mathbf{0.040}$ \textbf  s \\
\hline
$22$ & $18$ & $8448$ & $133.390$ s  & $\mathbf{0.020}$ \textbf  s \\
\hline
$23$ & $19$ & $5400$ & $20.590$ s  & $\mathbf{0.010}$ \textbf  s\\
\hline
$24$ & $20$ & $2160$ & $9.120$ s  & $\mathbf{0.010}$ \textbf  s\\
\hline
$25$ & $21$ & $576$ & $0.470$ s  & $\mathbf{0.010}$ \textbf  s\\
\hline
\hline
$27$ & $23$ & $432$ & $0.180$  s & $\mathbf{0.010}$ \textbf  s \\
\hline
$28$ & $24$ & $72$ & $\mathbf{0.000}$  \textbf  s & $\mathbf{0.000}$ \textbf  s \\
\hline
\end{tabular}
\end{table}

\begin{table}[ht]
\caption{Comparison time to compute the number of minimum-weight codewords of $C_{2q^2-2q-2}$  with $q=4$, using Algorithm \ref{Alg} and MAGMA command \texttt{MinimumWords}.} \label{Tab.q4.3fase} 
\centering 
\begin{tabular}{|c|c|c|c|c|}
\hline
$m$ & $d$ &   $\texttt{MW}_m$ & Algorithm \ref{Alg} & \texttt{MinimumWords}\\
\hline\hline
$22$ & $12$ & $150000$ & $\mathbf{85.87}$ \textbf s & Termination predicted at  $10^{16}$ s ($10^8$ y) \\
\hline
\end{tabular}
\end{table}

Note that for the first codes of the third phase, namely the codes $C_m$ having lowest $m$, our algorithm  is more convenient  than  MAGMA command  \texttt{MinimumWords}. In fact, it improves the running times needing to compute the number of minimum-weight codewords. 
A striking case is that of $q=4$ and $m=22$   (Table \ref{Tab.q4.3fase}), where our algorithm can get the wanted number in  $85.87$ \ecr seconds while \texttt{MinimumWords}  declares that termination requires $10^8$ years.
The fact that our algorithm  shows  a better performance \ecr for low values of $m$ bodes well for the codes of the second phase.

\section{Conclusions and further research}\label{Sec.con}

The keystone to finding a geometrical characterization of any minimum-weight codewords for the third and fourth phase was the choice of a different term ordering with respect to what can be usually found in literature, that is the $\mathtt{DegRevLex}$ with $y>x$.

The main novel aspect of our approach was indeed in this  deeper attention   to the  geometry  of the problem.  For instance, we constructed   monomial bases for the quotient ring $A_E$ and special subsets of these bases gave rise to the   parity-check matrices of the Hermitian codes. These sets of monomials are intrinsically connected to the Hermitian codes themselves, hence we can find them  in most papers on the topic. However,  the methods used to work with these sets of monomials can be  strikingly different. For instance  in  \cite{CGC-cd-art-marcolla2015small} the authors deal with  the problem of computing the number of minimum-weight codewords for   the  Hermitian codes with distance $d\leq q$ exploiting  peculiar  methods in linear algebra, as Vandermonde matrices.  They find that the minimum-weight codewords are complete intersection of $\He$ and a line of either type $y=ax+b$ or $x=a$. The methods used in \cite{CGC-cd-art-marcolla2015small} appear to be not sufficient to deal with codes with distances larger than the ones considered in their article, as they state in the conclusion of the paper.

\lq\lq \textit{As regards the other phases, it seems that only a part of the second phase
	can be described in a similar way. Therefore, probably a radically different
	approach is needed for phase-3, 4 codes in order to determine their weight distribution completely. Alas, we have no suggestions as to how reach this.}\rq\rq

In \cite{marcollaroggero2016minimum} we have tested our new geometrical approach for codes of phases 1-2, both from a theoretical and a computational point of view, while in the present paper we have tested it for codes of phases 3-4, thus developing the \emph{radically different approach} sought-after in  \cite{CGC-cd-art-marcolla2015small}.
We are also confident that this could be a powerful tool also to study the weights distribution of Hermitian codes and to design an efficient algorithm of error detection and correction.

Moreover, there are evidences that complete intersections can provide a good description of codewords with a small weight.  \\
\ecr


\section*{Acknowledgements}

The authors would like to thank M. Sala for the discussions and his comments.\\

\section*{Bibliography}
\bibliographystyle{amsplain}

\bibliography{BibChiMarg_giugno2018}

\end{document}